\documentclass[12pt]{article}
\usepackage{graphicx}
\usepackage{color}
\usepackage{transparent}
\usepackage{amsmath,amsfonts,amssymb}
\usepackage[noend]{algpseudocode}
\usepackage{algorithm}
\usepackage{fullpage}
\usepackage{hyperref}
\usepackage{tikz}
\usepackage{amsthm}

\newcommand{\suppress}[1]{}

\newtheorem{theorem}{Theorem}[section]
\newtheorem{lemma}[theorem]{Lemma}

\theoremstyle{definition}
\newtheorem{definition}{Definition}[section]

\newcommand{\norm}[2][]{\ensuremath{\left\Vert #2 \right\Vert_{#1}}}

\floatname{algorithm}{Algorithm}

\newcommand\xx{\boldsymbol{\mathit{x}}}

\newcommand\rr{\boldsymbol{\mathit{r}}}
\newcommand\yy{\boldsymbol{\mathit{y}}}

\newcommand\PP{\boldsymbol{\mathit{P}}}

\newcommand\ww{\boldsymbol{\mathit{w}}}
\newcommand\zz{\boldsymbol{\mathit{z}}}
\newcommand\sss{\boldsymbol{\mathit{s}}}

\begin{document}

\title{Rock-Paper-Scissors, Differential Games and \\
Biological Diversity}

\author{Tung Mai\\Georgia Tech
\and Ioannis Panageas\\MIT 
\and Will Ratcliff\\Georgia Tech
\and Vijay V. Vazirani\footnote{On leave from Georgia Tech.}\\UC Irvine
\and Peter Yunker\\Georgia Tech}

\date{}
\maketitle

\abstract{
We model a situation in which a collection of species derive their fitnesses via a rock-paper-scissors-type game; however, the precise payoffs are
a function of the environment. The new aspect of our model lies in adding a feedback loop: the environment changes according to the relative fitnesses
of the species; in particular, it gives a boost to the species having small populations. We cast our model in the setting of a differential game and 
we show that for a certain setting of parameters, this dynamics cycles. Our model is a natural one,
since depletion of resources used by more frequent species will shift the payoff matrix towards favoring less frequent ones. 
Since the dynamics cycles, no species goes extinct and diversity is maintained.

\section{Introduction}

Game theory has yielded deep insights into biological phenomena for almost a century. For example, the work of Fisher, Haldane, and Wright gave the central model of replicator dynamics which has been used extensively to
study evolutionary processes. Over the last few years, algorithmic game theory has contributed substantial ideas to this field.
In particular, Chastain et. al. \cite{PNAS2:Chastain16062014} showed that replicator dynamics can be viewed as the highly versatile algorithm of multiplicative weights update (MWU) and a sequence of papers used this insight to study the process of evolution \cite{PNAS2:Chastain16062014,evolfocs14,ITCS15MPP,MPPTV15,cacm}.
%
%

In most biological phenomena, the fitness of strategies is frequency-dependent and fixed payoff matrices are insufficient to describe them. One way of
understanding some of these phenomena is via dynamic games, in which the payoff matrix can change with time. 
Two main categories of such games are stochastic games \cite{Shapley53} which are discrete, with the payoff matrix being governed by a Markov chain, and differential games \cite{rufus}, the state space of which is described via a differential equation (continuous time dynamical system). An example of the 
first is \cite{MPPTV15}, which studies evolution via the MWU algorithm, and an example of the second is the recent elegant model of Weitz et al. \cite{weitz}, 
studying a dynamical version of the tragedy of the commons.

The work presented here is inspired by the last paper, \cite{weitz}. A tragedy of the commons occurs when a large number of agents simultaneously play a 
Prisoner's Dilemma-type game.
Selfish behavior on the part of agents, i.e., defection, leads to a highly sub-optimal outcome as compared to the globally best outcome
in which all agents cooperate. An early, and motivating example, of this
phenomena was cattle grazing in a common pasture. Each shepherd's selfish strategy of letting his cattle overgraze leads to
a depleted pasture land. The novel idea in \cite{weitz} was
to model a situation in which the actions of the agents slowly change the environment, and hence the payoff matrix, so that the selfish strategy 
of agents itself changes.
They give differential equations modeling this situation and show regimes of parameters under which the game cycles between the 
two extremes of replete and depleted environments, with selfish behavior cycling between cooperation and defection.
%
%

The underlying game analyzed by Weitz et. al. is particularly simple in that each agent has only two strategies.
In this paper, we consider a more complex class of games which are non-transitive, namely rock-paper-scissors, which has three strategies, and its generalization to
$n$ strategies. The rock-paper-scissors (RPS) game belongs to a general class of negative feedbacks in biology caused by non-transitivity. 
In such scenarios, there is no global optimum, as each strategy beats one of the other two, and is beaten by a third. 
Non-transitive dynamics appear to be widespread in biology, and rock-paper-scissors games have been reported in 
diverse organisms, including plants \cite{bio1, bio2}, animals \cite{bio3}, and microorganisms \cite{bio4,bio5,bio6}. In general, rock paper 
scissors dynamics maintain biological diversity (i.e., individuals, genes, or species), as no single strategy 
is capable of outright dominance \cite{bio1, bio4, bio7}. Most prior models of RPS dynamics used in biology 
%
%
assumed fixed payoff matrices. This is an incongruity with biology, where the fitness of most strategies will be frequency-dependent. 
On the other hand, analyzing complex games, such as RPS, in the setting of differential games is not straightforward and there is a need to develop
techniques for doing this.

%
%

Our model is as follows: We assume that the fitness of a collection of species is given by a rock-paper-scissors-type game. Additionally, the precise payoff
matrix is a function of the environment. The new aspect of our model lies in adding a feedback loop: the environment changes according to the relative fitnesses
of the species; in particular, it gives a boost to the species having smaller populations. We cast our model in the setting of a differential game and 
we show that for a certain setting of parameters, this dynamics cycles. Our model is a natural one,
since depletion of resources used by more frequent species will shift the payoff matrix towards favoring less frequent ones. 
Since the dynamics cycles, no species goes extinct and diversity is maintained.

More precisely, assume a population of $n$ species and let $\xx$ give the relative population of each species at any point in time, i.e., $x_i$ gives
the fraction of population of species $i$. We define $\PP_1, \ldots , \PP_n$ to be $n$ RPS-type
payoff matrices, where $\PP_i$ favors species $i$ over other species. At any point in time, the payoff matrix $\PP$ is a convex combination of these $n$
matrices given by weights $\ww$. We now define two replicator dynamics:
\begin{enumerate}
	\item The population $\xx$ follows a replicator dynamics based on payoff matrix $\PP$.
	\item The weights $\ww$ change according to a second replicator dynamics, based on $\xx$, in such a way that for smaller population species $i$, the 
    weight of $\PP_i$ tends to increase, hence giving species $i$ a boost.
\end{enumerate}

\smallskip
\noindent
\textbf{Our techniques:} We analyze a modification of the model of Weitz et al. \cite{weitz}. In our setting, the payoff matrix corresponds to an RPS game for $n\geq 3$ species (strategies). We show that the dynamics as defined in (\ref{eq:dynamics}) ``cycle" in the following sense: For all but a measure zero of initial population vectors $\xx_0 >0$ and weight vectors $\ww_0>0$ (strictly positive coordinates) and for any $\epsilon >0$, the trajectory of (\ref{eq:dynamics}) will return to a distance at most $\epsilon$ from $(\xx_0,\ww_0)$ infinitely often (for an illustration of the theorem, see Figures \ref{fig:x},\ref{fig:w}). The proof of the main theorem relies on the Poincar\'e recurrence theorem \ref{thm:poinca}, a well known theorem established a century ago. In words, the theorem says that systems that satisfy the two conditions of conservation of volume and that no orbit goes to the boundary will, after finite time, return to a state very close to the initial state and this will happen an infinite number of times. We proved that  under a certain homeomorphic transformation $\Pi$ (see Definition \ref{def:transform}), our dynamics satisfies both conditions of Poincar\'e's theorem. 

We established the second condition (Lemma \ref{lem:bounded}) by coming up with a log-barrier function (a similar idea is used in constrained optimization) which is finite in the interior of the state space and infinite at the boundary and by proving that this function is constant with respect to time. This suffices to prove that the orbits do not reach the boundary of $\Delta_n \times \Delta_n$ (and hence the orbits for the transformed system under $\Pi$ are bounded). The former condition can be proved via Liouville's theorem \ref{thm:liouville} (see also \cite{Sandholm10}). Finally, it can be shown that $\Pi^{-1}$ exists and is continuous, hence the result for the transformed dynamics carries over to the original dynamics (\ref{eq:dynamics}). Our results hold for any dimensions in contrast to \cite{weitz}; their results hold for 2 dimensions ($2\times 2$ payoff matrices) only.

We believe that these kinds of techniques will be useful for future models of non-transitive dynamics. We also believe that such models will
benefit from a study within algorithmic game theory.

\suppress{

\section{Related Work}
Evolutionary dynamical systems are central to the sciences due to their versatility in modeling a wide variety of biological, social and cultural phenomena \cite{NowakBook}.
Such dynamics are often used to capture the deterministic, {\em infinite population} setting, and are typically the first step in our understanding of seemingly complex evolutionary processes. The mathematical introduction of (infinite population) evolutionary processes is dating back to the work of Fisher, Haldane, and Wright in the beginning of the twentieth century. One example is the replicator equations (see Section \ref{sec:replicator}).

Theoretical computer science community is particularly interested in evolutionary dynamics. Valiant \cite{DBLP:journals/jacm/Valiant09} started viewing evolution through the lens of
computation and after that we have witnessed an accumulation of papers and problem proposals \cite{evolfocs14, DSV12, PNAS2:Chastain16062014}. 
We suggest the reader to read this nice expository \cite{cacm}. Later on papers appear where techniques from dynamical systems are used \cite{ITCS15MPP, PSV15, PV16, PapV16, pappil16}.

\medskip
A paragraph to add here!!
Change of environment \cite{wolf2005diversity}, cite a couple of papers on stochastic and differential games (maybe sequential too).

}

\subsection{Model}
\subsubsection{Classic RPS}
RPS is a game between two players, each of whom has three strategies. The payoff matrix can be written as
\[
\PP = \begin{bmatrix}
0 & -1 & 1 \\
1 & 0 & -1 \\
-1 & 1 & 0 \\
\end{bmatrix}.
\]
The RPS game can be generalized to more than three strategies, with the $n \times n$ payoff matrix being:
\[
\PP = \begin{bmatrix}
0 & -1 & 0 & 0  & \dots & 0 & 0 & 1 \\
1 & 0 & -1 & 0  & \dots & 0 & 0 & 0 \\
\vdots & \vdots & \vdots & \vdots &\vdots &\vdots & \vdots & \vdots\\
0 & 0 & 0 & 0 &\dots & 1 & 0 & -1\\
-1 & 0 & 0 & 0 & 0 & \dots & 1 & 0\\
\end{bmatrix}.
\]
In the context of biology, we will have $n$ competing species corresponding to the $n$ strategies, 
Let $\xx$ be the population vector where $x_i$ is the fraction of the population that is species $i$  
and $\rr$ be the fitness vector. Note that $\rr$ can be computed given $\xx$, i.e.,
$ \rr = \PP \xx.$ The replicator dynamics under $\PP$ is
\[x_i' = x_i(r_i - \overline{r}) \quad \forall i,\]
where $ \overline{r} = \xx^{\top}\PP\xx$ is the average fitness (if $\PP$ is antisymmetric, $\overline{r}$ is zero).

We will call the above the \textit{static} setting, since $\PP$ is fixed. Replicator dynamics in zero sum games have already been analyzed, proving that it has limit cycles or is recurrent (e.g., see \cite{Akin84, Sandholm10, sodapiliouras, piliourasaamas}).

\subsubsection{Dynamic payoff matrix}
We next define a \textit{dynamic} setting, where at all times the payoff matrix is a convex combination of $n$ matrices:
\[ \PP^{\ww} = w_1 \PP_1 + w_2 \PP_2 + \dots + w_n \PP_n, \]
where
\[
\PP_1 = \PP + \begin{bmatrix}
0 & \mu & \dots & \mu & \mu \\
-\mu & 0 & \dots & 0 & 0 \\
\vdots & \vdots & \vdots & \vdots & \vdots\\
-\mu & 0 & \dots & 0 & 0 \\
\end{bmatrix},
\PP_2 = \PP + \begin{bmatrix}
0 & -\mu & 0 & \dots & 0 \\
\mu & 0 & \mu & \dots & \mu \\
0 & -\mu & 0 & \dots & 0 \\
\vdots & \vdots & \vdots & \vdots & \vdots\\
0 & -\mu & 0 & \dots & 0 \\
\end{bmatrix},
\dots ,\]
\[
\PP_n = \PP + \begin{bmatrix}
0 & 0 & \dots & 0 & -\mu \\
0 & 0 & \dots & 0 &-\mu \\
\vdots & \vdots & \vdots & \vdots & \vdots\\
0 & 0 & \dots & 0 &-\mu \\
\mu & \mu & \dots & \mu & 0 \\
\end{bmatrix}.
\]
Formally, $\PP_i-\PP$ is a matrix with entries $\mu$ at the $i$-th row and $-\mu$ at the $i$-th column (rest of the entries and diagonal entry $(i,i)$ are zero). It can be seen that for $\mu>0$ (for the rest of the paper $\mu \geq 0$), $\PP_i$ favors type $i$ by increasing the payoff of $i$ when competing with other types. The weight vector $\ww = (w_1, \dots, w_n)$ changes according to the replicator dynamics under the following matrix
\[
\AA = \begin{bmatrix}
0 & x_2 - x_1 & x_3 - x_1 & \dots & x_{n-1}-x_1 & x_n-x_1 \\
x_1 - x_2 & 0 & x_3 - x_2 & \dots & x_{n-1}-x_2 & x_n-x_2 \\
\vdots & \vdots & \vdots & \vdots & \vdots & \vdots \\
x_1 - x_n & x_2 - x_n & x_3 - x_n & \dots & x_{n-1}- x_n & 0 \\
\end{bmatrix}.
\]
In words, $\PP^{\ww}$ changes with time. The idea behind the way $\ww$ changes, is that species with small population should be favored. Since $\AA$ is also anti-symmetric and has 0 entries in the diagonal, the replicator dynamics update rule of $\ww$ is
\[ w_i' = w_i\sum_{j} w_j (x_j - x_i)  \quad \forall i.\]

By model's definition,
\[
\PP^{\ww} = \begin{bmatrix}
0 & -1+ \mu(w_1-w_2) &  \mu(w_1 - w_3) & \dots & 1+\mu(w_1 - w_n) \\
1+\mu(w_2-w_1) & 0 & -1+\mu(w_2-w_3) & \dots & \mu(w_2 - w_n) \\
\vdots & \vdots & \vdots & \vdots & \vdots\\
\mu(w_{n-1} - w_1) & \mu(w_{n-1} - w_2) & \mu(w_{n-1} - w_3) & \dots  & -1+\mu(w_{n-1} - w_n)\\
-1+\mu(w_{n} - w_1) & \mu(w_{n} - w_2) & \mu(w_{n} - w_3) &\dots & 0 \\
\end{bmatrix}.
\]
Note that $\PP$ is anti-symmetric and has 0 entries in the diagonal. The replicator dynamics of $\xx$ becomes
\[ x_i' = x_i \cdot s_i  \quad \forall i,\]
where
$ \sss = \PP^{\ww} \xx$ is the fitness vector of population $\xx$ under $\PP^{\ww}$. Summing up, the system of ordinary differential equations that we would like to analyze is captured by ($w_i$ depends on $x_i$)

\begin{equation}\label{eq:dynamics}
x_i' = x_i \sum_j \PP^{\ww}_{ij} x_j, \; w_i' = w_i \sum_{j} w_j (x_j - x_i)  \quad \forall i.
\end{equation}

\smallskip
\noindent
\textbf{Notation:} We denote the probability simplex on a set of size $n$ as $\Delta_n$. Vectors in $\mathbb{R}^n$ are denoted in bold-face letters $\xx$ and are considered as column vectors. The $i$-th coordinate is denoted by $x_i$. To denote a row vector we use
$\xx^{\top}$. The time derivative of a function $y = y(t)$ is denoted by $y'$.
\section{Preliminaries}
\subsection{Dynamical Systems}
Let $f:\mathcal{S} \to \mathcal{S}$ be continuously differentiable with $\mathcal{S} \subset \mathbb{R}^n$, $\mathcal{S}$ an open set. A \emph{continuous (time) dynamical system} is of the form
\begin{equation}\label{eq:continuous}
\yy' \equiv \frac{d \yy}{dt} = f(\yy).
\end{equation}
Since $f$ is continuously differentiable, the system of ordinary differential equations (ode (\ref{eq:continuous})) along with the initial condition $\yy(0) = \yy_0 \in \mathcal{S}$ has a unique solution for $t \in \mathcal{I}(\yy_0)$ (some time interval) and we can present it by $\phi(t,\yy_0)$, called the \emph{flow} of the system. $  \phi_{t}(\yy_0) \equiv \phi(t,\yy_0)$ corresponds to a function of time which captures the \emph{trajectory/orbit} of the system with $\yy_0$ the given starting point. It is continuously differentiable, its inverse exists (denoted by $\phi_{-t}(\yy_0)$) and is also continuously differentiable (called \emph{diffeomorphism}) in the so called maximal interval of existence $\mathcal{I}$. It is also true that $\phi_{t}\circ \phi_{s} =\phi_{t+s}$ for $t,s,t+s \in \mathcal{I}$. $\yy_0 \in \mathcal{S}$ is called an \emph{equilibrium} if $f(\yy_0)=\textbf{0}$. In that case holds $\phi_t(\yy_0) = \yy_0$ for all $t\in \mathcal{I}$, i.e., $\yy_0$ is a \emph{fixed point} of the function $\phi_{t}(\yy)$ for all $t\in \mathcal{I}$.

 If $f$ is globally Lipschitz then the flow is defined for all $t \in \mathbb{R}$, i.e., $\mathcal{I} = \mathbb{R}$. One way to enforce the dynamical system to have a well-defined flow for all $t \in \mathbb{R}$ is to renormalize the vector field by $\norm[]{f(\yy)}+1$, i.e., the resulting dynamical system will be $\frac{d \yy}{dt} = \frac{f(\yy)}{\norm[]{f(\yy)}+1}$, because the function becomes globally $1$-Lipschitz.
The two dynamical systems (before and after renormalization) are topologically equivalent (\cite{perko}, p. 184). Formally this means that there exists a homeomorphism $H$ which maps trajectories of (\ref{eq:continuous}) onto trajectories of the renormalized flow and preserves the direction of time. In words it means that the two systems have the same behavior/geometry (same fixed points, convergence properties, phase portrait). For the rest of the paper, we may assume that the flow of dynamics (\ref{eq:dynamics}) is well-defined for all $t \in \mathbb{R}$.

\begin{definition}[\textbf{Volume preserving}]
The differential equation (\ref{eq:continuous}) is said to be volume preserving on $\mathcal{S}$ if for any measurable set $B \subseteq \mathcal{S}$, we have $\nu (\phi_t(B)) = \nu(B)$ for all $t \in \mathbb{R}$, where $\nu$ is the Lebesgue measure and $\phi_t$ the flow of the ode.
\end{definition}
The most common way to prove that the flow is volume preserving is via the Liouville's theorem.
\begin{theorem}[\textbf{Liouville theorem} \cite{Sandholm10}]\label{thm:liouville}
Let $\yy' = f(\yy)$ be an ode with flow $\phi_t$. It holds that
\[
\frac{d \nu (\phi_t(B))}{dt} = \int_{\phi_t(B)} (\nabla \cdot f) d\nu, \textrm{ for each Lebesgue measurable set }B.
\]
Therefore, as long as $\nabla \cdot f=0$, the flow preserves volume.
\end{theorem}
In 1890, Poincar\'e \cite{poincare} showed that whenever a dynamical system preserves volume, almost all trajectories return arbitrarily close to their initial position an infinite number of times.
\begin{theorem}[\textbf{Poincar\'e Recurrence} \cite{poincare,bar}]\label{thm:poinca} If a flow preserves volume and has only bounded orbits then for each open set there exist orbits that intersect the set infinitely often.
\end{theorem}

For more information on dynamical systems see \cite{perko} and for readers familiar with game theory see \cite{Hofbauer98, Sandholm10}.

\subsection{Replicator Dynamics and Evolution}\label{sec:replicator}
Replicator equations, first were introduced by Fisher \cite{Fisher30} in 30's for genotype evolution. The simplest form of replicator equations is the following:
\begin{equation*}
x_i'(t) = x_i(t)((A\xx(t))_i- \xx(t)^{\top}A\xx(t)).
\end{equation*}
where $A$ is a payoff matrix (generally non-negative), $\xx$ a vector that lies in simplex and $(A\xx)_i$ denotes $\sum_j A_{ij}x_j$. Observe that in the nonlinear dynamics above, simplex is invariant (if we start from a probability distribution, the vector remains a probability distribution). This dynamics is called replicator dynamics and has been used numerous times in biology, evolution, game theory and genetic algorithms. The dynamics we analyze in this paper is a version of replicator dynamics (on generalized RPS with dynamics payoff matrix).

\section{Our results}
In this section we state and prove our main result. Our main theorem can be stated formally as follows (see also Figures \ref{fig:x},\ref{fig:w}):
\begin{theorem}\label{thm:main} For all but measure zero of initial positions, the trajectories of the dynamics (\ref{eq:dynamics})  return arbitrarily close to their initial position an infinite number of times.
\end{theorem}
\begin{definition}[\textbf{Natural transformation} \cite{Hofbauer98}]\label{def:transform} We define the natural transformation $\Pi : \textrm{int}(\Delta_n) \to \mathbb{R}^{n-1}$ to be $\Pi(\yy) = \left(\log \left (\frac{y_1}{y_n}\right),...,\log \left (\frac{y_{n-1}}{y_n}\right)\right)$. In words, we map every point $\yy$ of the interior of simplex to a point in $\mathbb{R}^{n-1}$. It is not hard to see that the map is injective and surjective. The reason is that $\Pi^{-1}$ exists and is equal to $\Pi^{-1}(\zz) = \left(\frac{e^{z_1}}{1+\sum_{j=1}^{n-1}e^{z_j}},\dots,\frac{e^{z_{n-1}}}{1+\sum_{j=1}^{n-1}e^{z_{j}}},\frac{1}{1+\sum_{j=1}^{n-1}e^{z_j}}\right)$.
The points on the boundary of simplex correspond to vectors with infinity Euclidean norm in $\mathbb{R}^{n-1}$.
\end{definition}
\begin{lemma}\label{lem:bounded} Let $(\xx^0,\ww^0)$ be an initial point in the interior of $\Delta_n \times \Delta_n$. The dynamics mapped to $\mathbb{R}^{n-1} \times \mathbb{R}^{n-1}$ (under natural mapping $(\Pi, \Pi)$) has bounded orbits $(\Pi(\xx^t),\Pi(\ww^t))$.
\end{lemma}
\begin{proof}
We define the log-barrier function
\[ D(\xx) = \sum_{i=1}^n \log \left( \frac{1}{x_i}\right). \]
Computing the derivative we get,
\begin{align*}
D'(\xx) &= - \sum_{i} \frac{x_i'}{x_i} = -  \sum_i s_i \\
 &= - \left( \sum_i x_i \mu  \left (\sum (w_j - w_i) \right) \right) \\
 &= -\mu \left(\sum_ i x_i (1 - n w_i)\right) \\
 &= -\mu \left( 1 - n \sum_i w_i x_i \right).
\end{align*}
We also define the log-barrier function
\[ D(\ww) = \sum_{i=1}^n \log \left( \frac{1}{w_i}\right), \]
and we have
\begin{align*}
D'(\ww) &= - \sum_{i} \frac{w_i'}{w_i} = -  \sum_i \sum_{j} w_j (x_j - x_i) \\
&=  -\sum_j \sum_{i} w_j (x_j - x_i)\\
&=   -\sum_j w_j (n x_j - 1)\\
&= -\left( n \sum_{j} w_j x_j -1 \right). \\
\end{align*}

Therefore,
\begin{align*}
D'(\xx)  + \mu D'(\ww) &= -\mu \left( 1 - n \sum_i w_i x_i \right) -\mu \left( n \sum_{i} w_i x_i -1 \right)\\
&= \left( n\sum_i w_i x_i - 1  \right) (\mu-\mu) = 0.
\end{align*}

Hence,
\[D(\xx)  + \mu D(\ww) \]
is a constant motion of time (independent of time). It is also clear that $D(\xx) , D(\ww) \geq 0$ and become infinity only on the boundary of $\Delta_n \times \Delta_n$. Since $D(\xx^0)+ \mu D(\ww^0)$ is bounded (i.e., $(\xx^0,\ww^0)$ is an interior point), we get that $D(\xx^t) +\mu D(\ww^t)$ is bounded for all times $t$, hence there is no subsequence of times $t_k$ so that $(\xx^{t_k},\ww^{t_k})$ converges to the boundary of $\Delta_n \times \Delta_n$. Therefore there is no subsequence of times $t_k$ so that $\norm[2]{(\Pi(\xx^{t_k}),\Pi(\ww^{t_k}))}$ goes to infinity, i.e., the dynamics on $\mathbb{R}^{n-1} \times \mathbb{R}^{n-1}$ (under natural mapping $(\Pi,\Pi)$) has bounded orbits.
\end{proof}
\begin{proof}[Proof of Theorem \ref{thm:main}] Under the mapping $\Pi$ for $\xx,\ww$, set $\yy \equiv \Pi(\xx)$ and $\zz \equiv \Pi (\ww)$. We shall show that the dynamics $(\yy' ,\zz') = g(\yy,\zz)$ satisfies the conditions of the Poincar\'e recurrence theorem (Theorem \ref{thm:poinca}), where $g$ is the vector field of the resulting dynamics, after the transformation $(\Pi,\Pi)$.

The vector field becomes $y_i' =  \frac{x_n}{x_i} \frac{x_i'x_n - x_ix_n'}{x_n^2} = \frac{x_i'}{x_i} - \frac{x_n'}{x_n} = \sum_j(P^{\ww}_{ij}- P^{\ww}_{nj})x_j = x_1+x_{i-1}-x_{n-1}-x_{i+1}+ \mu (w_i-w_n)$ (with the convention that $x_0 = x_n$). Similarly $z_i' =  \frac{w_n}{w_i} \frac{w_i'w_n - w_iw_n'}{w_n^2} = \frac{w_i'}{w_i} - \frac{w_n'}{w_n} = \sum_jw_j (x_n - x_i) = x_n - x_i$.

\smallskip
After substitution (using $\Pi^{-1}$) we get
\begin{equation}\label{eq:newdynamics}
y_i' = \frac{e^{y_1}+e^{y_{i-1}}-e^{y_{n-1}}-e^{y_{i+1}}}{1+\sum_{j=1}^{n-1}e^{y_j}}+\mu\frac{e^{z_i}-1}{1+\sum_{j=1}^{n-1}e^{z_j}},\; z_i' = \frac{1-e^{y_i}}{1+\sum_{j=1}^{n-1}e^{y_j}}.
\end{equation}
We shall prove that the flow of the dynamics (\ref{eq:newdynamics}) preserves the volume, by showing that $\nabla \cdot g =0 $ (then it follows from Liouville's theorem \ref{thm:liouville}). 

Set $S = \sum_{j=1}^{n-1}e^{y_j}$. We compute the partial derivatives and we get that 

\begin{equation}
\begin{array}{ll}
\frac{\partial g}{\partial y_1} &= \frac{e^{y_1}(1+S) - e^{y_1}(1+e^{y_1}-e^{y_2}-e^{y_{n-1}})}{(1+S)^2},\\
\frac{\partial g}{\partial y_j} &= -\frac{e^{y_j}(  e^{y_1}+e^{y_{j-1}}-e^{y_{j+1}}-e^{y_{n-1}})}{(1+S)^2}, \textrm{ for }2 \leq j\leq n-2,\\
\frac{\partial g}{\partial y_{n-1}} &= \frac{-e^{y_{n-1}}(1+S) - e^{y_{n-1}}(e^{y_1}+e^{y_{n-2}}-1-e^{y_{n-1}})}{(1+S)^2},\\
\frac{\partial g}{\partial z_j} &= 0 \textrm{ for }1 \leq j\leq n-1.
\end{array}
\end {equation}
Therefore we get that 
\begin{align*}
\nabla \cdot g &= \sum_{j=1}^{n-1} \frac{\partial g}{\partial y_j} + \sum_{j=1}^{n-1}\frac{\partial g}{\partial z_j} \\&= \frac{-\sum_{j=1}^{n-1} e^{y_j}(e^{y_1}+e^{y_{j-1}}-e^{y_{j+1}}-e^{y_{n-1}})+(e^{y_1}-e^{y_{n-1}})(1+\sum_{j=1}^{n-1}e^{y_j})}{(1+S)^2},
\end{align*}
with the convention that $y_0=y_n=0$. 
We get that all the terms of the form $e^{y_j}e^{y_{j+1}}$ cancel out (telescopically) and also terms $e^{y_1}e^{y_j}, e^{y_{n-1}}e^{y_j}$ also cancel out for all $j$, i.e., it turns out that $\nabla \cdot g=0$. Hence we conclude from Liouville theorem that the dynamical system (\ref{eq:newdynamics}) preserves volume.

Using the fact that the flow preserves the volume and that the orbits are bounded (Lemma \ref{lem:bounded}), we apply Poincar\'e recurrence theorem on $(\tilde{\yy},\tilde{\zz}) = g(\yy,\zz)$ for a small open ball around any initial point $(\tilde{\yy}_0,\tilde{\zz}_0)$ and the claim follows for the transformed dynamics. Since $\Pi^{-1}$ is continuous, if the distance between two points goes to zero in $\mathbb{R}^{n-1}\times \mathbb{R}^{n-1}$, so it does in $\Delta_{n}\times \Delta_n$ and the claim follows for the dynamics (\ref{eq:dynamics}).
\end{proof}

\begin{figure}[htb!]
		\centering     
			\includegraphics[width=0.60\linewidth]{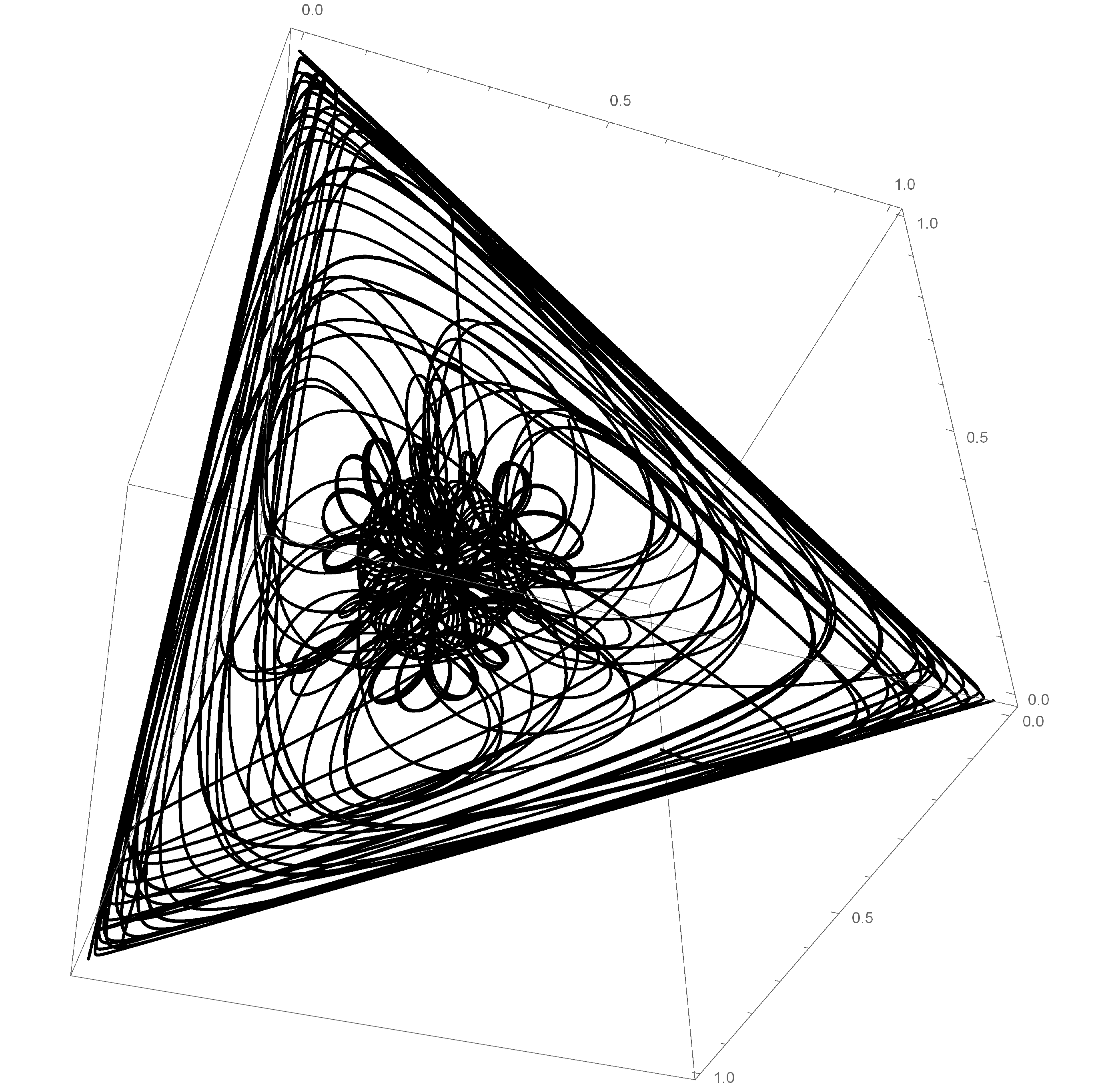}
			\caption{Trajectories of the vector $\xx$ for different initial positions with $\mu = 0.1$. Trajectories intersect due to the fact 6 dimensions are projected to a 3D figure. The ``cycling" behavior is observed.}
			{\label{fig:x}}
		\end{figure}
		\begin{figure}[htb!]
			\centering
			\includegraphics[width=0.60\linewidth]{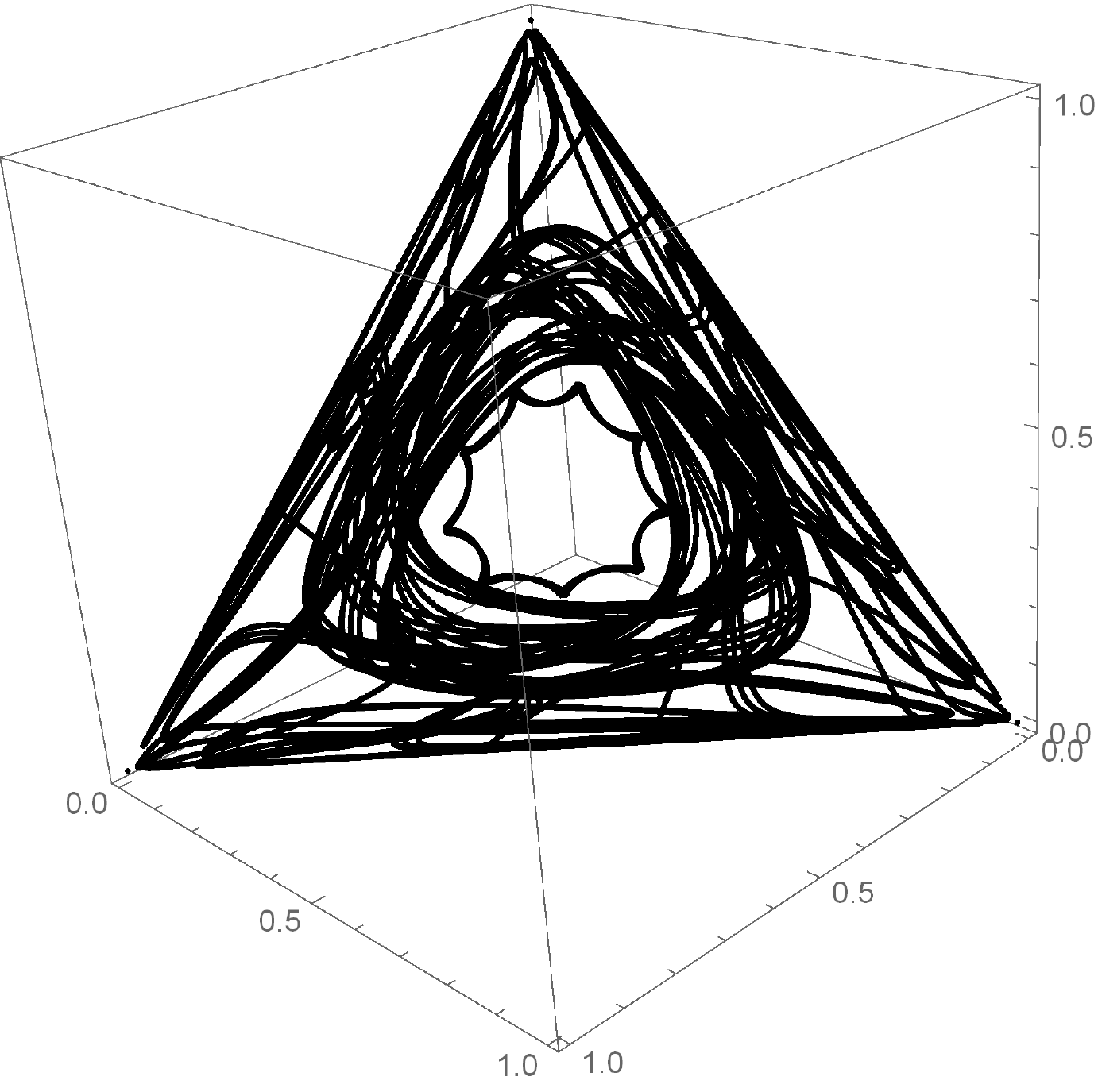}
			\caption{Trajectories of the vector $\ww$ for different initial positions with $\mu = 0.1$. Trajectories intersect due to the fact 6 dimensions are projected to a 3D figure. The ``cycling" behavior is observed.}
			{\label{fig:w}}
	
\end{figure}

\section{Discussion}
Let us first address some immediate generalizations of our model.  
Our proofs can be easily extended to the case where in the antisymmetric matrix $\PP$, each 1 and -1 pair is replaced by $a$ and $-a$, for a different parameter $a$.
We have presented the simpler case in this abstract in order to enhance clarity of the main new ideas. The full paper will contain the more general case.

There are several generalizations of our model which are worth studying. For example, what happens if each matrix $\PP_i$ has its own parameter $\mu_i$? So far we have not been able to prove cycling for this general case but simulations indicate that the system does cycle. Understanding the range of parameters which lead to cycling and those that do not will be interesting. 

We are not aware of other uses of the Poincar\'e theorem in differential games. It will be nice to see other applications of this powerful theorem  
to learning dynamics, other than replicator dynamics, where the state space (payoff matrix) changes as per a differential equation. 

More broadly, we believe that the work of \cite{weitz} and our work are opening up the possibility of modeling more complex biological phenomena using the setting of differential games and obtaining insights into new biological phenomena. 

\section*{Acknowledgements}
 
Ioannis Panageas would like to acknowledge a MIT-SUTD postdoctoral fellowship and thank Georgios Piliouras for pointing out \cite{sodapiliouras} and fruitful discussions.

\bibliographystyle{plain}
\bibliography{bibliography}

\end{document}